\documentclass[a4paper]{scrartcl}

\usepackage{amsfonts,amsthm,amssymb,amsmath}
\usepackage[percent]{overpic}
\usepackage[T1]{fontenc}
\usepackage{graphicx}
\usepackage[inline]{enumitem}
\usepackage{hyperref}
\usepackage{tikz}
\usepackage{tkz-berge}

\newtheorem{thm}{Theorem}
\newtheorem{lem}[thm]{Lemma}
\newtheorem{prp}[thm]{Proposition}
\newtheorem{cor}[thm]{Corollary}

\theoremstyle{definition}

\newtheorem{prb}[thm]{Problem}

\newcommand{\Cay}{\operatorname{Cay}}
\newcommand{\Aut}{\operatorname{Aut}}

\newcommand{\Z}{\ensuremath{\mathbb Z}}
\newcommand{\N}{\ensuremath{\mathbb N}}

\newcommand{\sdr}{spanning double ray}

\usepackage{authblk}

\title{Invariant spanning double rays in amenable groups}

\begin{document}
	\author[1]{Agelos Georgakopoulos\thanks{Supported by the European Research Council (ERC) under the European Union's Horizon 2020 research and innovation programme (grant agreement No 639046).}
	}
	\author[2]{Florian Lehner\thanks{Supported by the Austrian Science Fund (FWF), grant J 3850-N32.}
	}
	\affil[1,2]{{Mathematics Institute}\\
		{University of Warwick}\\
		{CV4 7AL, UK}}
	
	\maketitle
	
\begin{abstract}
A well-known result of Benjamini, Lyons, Peres, and Schramm	states that if $G$  is a finitely generated Cayley graph of a group $\Gamma$, then $\Gamma$ is amenable if and only if $G$  admits a $\Gamma$-invariant random spanning tree with at most two ends. We show that this is equivalent to the existence of a $\Gamma$-invariant random spanning double ray in a power of $G$.
\end{abstract}
	
	\section{Introduction}
	In this paper we study (Cayley) graphs $G$ having a random \sdr\  the distribution of which is invariant under the action of some subgroup of the automorphism group $\Aut G$ of $G$.
	
	Our first result is
	\begin{thm}
		\label{Z2}
		The edges of the square grid can be decomposed into a  random, $\Z^2$-invariant, pair of spanning double rays. Moreover, if $G$ is a connected locally finite Cayley graph of an Abelian group $\Gamma$, then $G$ admits a random $\Gamma$-invariant spanning double ray. 
	\end{thm}
	
	A {\em spanning double ray} is a connected subgraph containing all vertices in which every vertex has exactly two neighbours.
	The {\em square grid} is the Cayley graph of $\Z^2$ with respect to a minimal generating set.
	
	\medskip
	
		Theorem~\ref{Z2}  raises the question of which Cayley graphs  admit a random spanning double ray that is invariant under the action of the underlying group, see also Problem~\ref{probBPT} below. Benjamini et al.\ \cite{group-invariant-percolation-1999} showed that amenability is a necessary condition. The fact that the square grid admits a $\Z^2$-invariant random spanning double ray was well-known, an example is the Peano UST curve, see e.g.\ \cite{LyonsBook}. 
	
	Alspach \cite{AlspachHamDecomp} conjectured that every Cayley graph of any finite Abelian group admits a decomposition of its edge set into Hamilton cycles. By \cite{ErLePiDecomp} every Cayley graph of $\mathbb Z^d$ admits a decomposition of its edge set into spanning double rays, confirming Alspach's conjecture in spirit for this class of groups. In light of this and the first part of Theorem~\ref{Z2} it is  natural to ask whether every Cayley graph of a 1-ended  Abelian group admits a random decomposition into double rays which is invariant under the group action.

	It is an open problem \cite[Problem 3]{fleisch} whether every  connected, 1-ended, finitely generated Cayley graph has a \sdr, let alone an invariant one. But it is known that a \sdr\ can always be found, if we are allowed to extend the generating set $S$ into $S\cup S^2$ \cite{ThomassenG2}; see also \cite[Corollary 6]{fleisch}. The main result of this paper implies that if we extend $S$ further into $S\cup S^2 \cup S^3$, then we can guarantee the existence of an invariant random \sdr. It holds in a more general setup of groups acting on graphs:
	
	\begin{thm}
		\label{thm:invariantspanningdoubleray}
		Let $G$ be a countably infinite, connected, amenable graph, and let $\Gamma$ be a closed subgroup of $\Aut G$ that acts transitively on $G$. Then there is a $\Gamma$-invariant random spanning double ray in $G^3$.
	\end{thm}
	
	It was known that $G^3$ contains a  (deterministic) \sdr\ for every countable, connected, 1-ended \cite{sekanina} or 2-ended \cite{fleisch} graph $G$, and this fact motivated our work.
	
	A further source of motivation, and an important tool in the proof of Theorem~\ref{thm:invariantspanningdoubleray}, is a well-known result of Benjamini et al.\  \cite{group-invariant-percolation-1999} showing that amenable Cayley graphs admit invariant random spanning trees with 1 or 2 ends. In fact they proved a much more general statement (\cite[Theorem~5.3]{group-invariant-percolation-1999}), that provides probabilistic characterisations of amenability. With Theorem~\ref{thm:invariantspanningdoubleray} we can add a further probabilistic characterisation, which combined with their statement gives the following:

	\begin{cor}[{Partly \cite[Theorem~5.3]{group-invariant-percolation-1999}}] \label{corol} 
		Let $\Gamma$ be a closed unimodular subgroup of $\Aut G$ that acts transitively on a countably infinite, locally finite, connected graph $G$. Then the following conditions are equivalent:
		\begin{enumerate}
			\item \label{itam} $G$ is amenable;
			\item \label{itst} there is a $\Gamma$-invariant random spanning tree of $G$ with at most 2 ends
			a.s.;
			\item \label{itpc} there is a $\Gamma$-invariant random nonempty connected subgraph $\omega$ of $G$
			that satisfies $p_c(\omega) = 1$ with positive probability;
			\item \label{itsdr} there is a $\Gamma$-invariant random spanning double ray in $G^3$;
			\item  \label{itk} there is a $\Gamma$-invariant random spanning double ray in $G^k$ for some $k\in \N$;
			\item \label{itamGr} $\Gamma$ is amenable.
		\end{enumerate}
	\end{cor}
	Our contribution to this is items \ref{itsdr}, \ref{itk}. If the unimodularity condition is dropped, then the following implications still hold: \ref{itam}  $\to$ \ref{itst}  $\to$ \ref{itpc}  $\to$ \ref{itamGr} \cite[Theorem~5.3]{group-invariant-percolation-1999}, and \ref{itst}  $\to$ \ref{itsdr}  $\to$ \ref{itk}  $\to$ \ref{itamGr}: indeed, \ref{itst}  $\to$ \ref{itsdr} is Theorem~\ref{thm:invariantspanningdoubleray}, and \ref{itk}  $\to$ \ref{itamGr} is obtained by applying the implication \ref{itst}  $\to$ \ref{itamGr} on $G^k$. We do not know whether the implication \ref{itk}  $\to$ \ref{itst} holds.

	
	\medskip
	We conjecture that  Theorem \ref{thm:invariantspanningdoubleray} can be strengthened by replacing $G^3$ with $G^2$, using the techniques of \cite{ThomassenG2,fleisch} when $G$ is 2-connected. (For this, one would need to produce an invariant ladder-like spanning structure to play the role of the random spanning tree we use in our proof.) The following might be much harder:
	
	\begin{prb}[{\cite[Question~4.5]{BePeTiPer}}] \label{probBPT}
		Does every amenable, 1-ended Cayley graph $G=\Cay(\Gamma,S)$ admit a $\Gamma$-invariant spanning double ray?
	\end{prb}
	
	We suspect that Corollary~\ref{corol} still holds when the action is quasi-transitive rather than transitive; in that case, the power $G^3$ would be optimal, as one could decorate $G$ by attaching an appropriate finite tree to each vertex. If $G$ is assumed to be 2-connected, then $G^2$ should be the optimal power.
	
	\section{Proof of Theorem \ref{Z2}}
	
	\begin{figure}
		\begin{minipage}{.5\textwidth}
			\centering
			\begin{tikzpicture}[scale=.6, vertex/.style={inner sep=1pt,circle,draw,fill}]
			\pgfdeclarelayer{E}
			\pgfdeclarelayer{V}
			\pgfdeclarelayer{drawunder}
			\pgfsetlayers{drawunder,E,V}
			\foreach \i in {-7,...,7}
			{
				\foreach \j in {-5,...,5}
				{
					\begin{pgfonlayer}{V}	
					\clip (-5.5,-3.5) rectangle (4.5,4.5);
					\node[vertex] at (\i,\j){};
					\end{pgfonlayer}
					\begin{pgfonlayer}{E}	
					\clip (-5.5,-3.5) rectangle (4.5,4.5);
					\path[draw] (\i,\j)--(\i+1,\j);
					\path[draw] (\i,\j)--(\i,\j+1);
					\end{pgfonlayer}
				}
				\begin{pgfonlayer}{drawunder}	
				\path[draw,very thick] (-1,3)--(0,3)--(0,0)--(-1,0)--cycle;
				\path[draw,very thick,dotted] (-2,2)--(-2,1)--(1,1)--(1,2)--cycle;
				\end{pgfonlayer}
				\foreach \i in {-3,...,3}
				{
					\begin{pgfonlayer}{drawunder}	
					\clip (-5.5,-3.5) rectangle (4.5,4.5);
					\path[draw,dashed,gray] (-8.5,-8.5+4*\i)--(7.5,7.5+4*\i);
					\path[draw,dashed,gray] (-8.5,7.5+4*\i)--(7.5,-8.5+4*\i);
					\end{pgfonlayer}
				}
			}
			\end{tikzpicture}
			\caption{Edge tiling of $\Z^2$.}
			\label{fig:tiles}
		\end{minipage}%
		\begin{minipage}{.5\textwidth}
			\centering
			\begin{tikzpicture}[scale=.6, vertex/.style={inner sep=1pt,circle,draw,fill}]
			\pgfdeclarelayer{E}
			\pgfdeclarelayer{V}
			\pgfdeclarelayer{drawunder}
			\pgfsetlayers{drawunder,E,V}
			\foreach \i in {-7,...,7}
			{
				\foreach \j in {-5,...,5}
				{
					\begin{pgfonlayer}{V}	
					\clip (-5.5,-3.5) rectangle (4.5,4.5);
					\node[vertex] at (\i,\j){};
					\end{pgfonlayer}
					\begin{pgfonlayer}{E}	
					\clip (-5.5,-3.5) rectangle (4.5,4.5);
					\path[draw] (\i,\j)--(\i+1,\j);
					\path[draw] (\i,\j)--(\i,\j+1);
					\end{pgfonlayer}
				}
				\begin{pgfonlayer}{drawunder}	
				\path[draw,very thick] (-1,3)--(0,3)--(0,1)--(2,1)--(2,-2)--(1,-2)--(1,0)--(-1,0)--cycle;
				\path[draw,very thick,dotted] (-2,2)--(-2,1)--(0,1)--(0,-1)--(3,-1)--(3,0)--(1,0)--(1,2)--cycle;
				\end{pgfonlayer}
				\foreach \i in {-3,...,3}
				{
					\begin{pgfonlayer}{drawunder}	
					\clip (-5.5,-3.5) rectangle (4.5,4.5);
					\path[draw,dashed,gray] (-8.5,-8.5+4*\i)--(7.5,7.5+4*\i);
					\path[draw,dashed,gray] (-8.5,7.5+4*\i)--(7.5,-8.5+4*\i);
					\end{pgfonlayer}
				}
			}
			\end{tikzpicture}
			\caption{Connecting two tiles.}
			\label{fig:connect}
		\end{minipage}
	\end{figure}
	
	We first outline a procedure to turn a one-ended spanning tree of $\Z ^2$ into a colouring of the edge set of $\Z^2$ in which each colour induces a spanning double ray. For this purpose, first consider a decomposition of the edge set of $\Z^2$ into tiles, see Figure~\ref{fig:tiles} (the bold solid and dotted edges together form one tile). Note that the subgroup $\Gamma$ of $\Z ^2$ generated by $(\pm 2, \pm 2)$ acts in a natural way on the tiles. Further, 2-colour the edges of each tile as shown in  Figure~\ref{fig:tiles}: the bold solid edges form one colour class, the bold dotted edges form the other).  Colour different tiles so that the colouring is invariant under the action of $\Gamma$.
	
	View the tiles as the vertices of the Cayley graph $\Cay(\Gamma, \{(\pm 2, \pm 2)\})$ of $\Gamma$, and note that any two adjacent tiles intersect in exactly two vertices. Call these two vertices the \emph{attachment vertices} of those tiles, and call the unique cycle of length $4$ containing both of them their \emph{attachment square}. Furthermore, every tile contains four edges of each colour that are not contained in any attachment square; call those the \emph{internal edges} of the tile. Observe that we can connect the colour classes in two adjacent tiles by swapping the colours within their attachment square, see Figure~\ref{fig:connect}. It is not hard to see that colour swaps at attachment squares can be used to connect  the colour classes of multiple squares as well. In fact, the following is true:
	
	\begin{prp}
		If $T$ is a one-ended spanning tree of $\Cay(\Gamma, \{(\pm 2, \pm 2)\}) \simeq \Z^2$, then swapping the colours at all attachment squares that correspond to edges of $T$ results in a colouring where each colour class comprises a spanning double ray.
	\end{prp}

\begin{figure}
\begin{minipage}{.5\textwidth}
\centering
			\begin{tikzpicture}[scale=.6, vertex/.style={inner sep=1pt,circle,draw,fill},treevertex/.style={inner sep=1.5pt,circle,draw,fill}]
			\pgfdeclarelayer{E}
			\pgfdeclarelayer{V}
			\pgfdeclarelayer{drawunder}
			\pgfdeclarelayer{drawover}
			\pgfsetlayers{drawunder,E,V,drawover}
			\foreach \i in {-7,...,7}
			{
				\foreach \j in {-5,...,5}
				{
					\begin{pgfonlayer}{V}	
					\clip (-5.5,-2.7) rectangle (4.5,5.3);
					\node[vertex] at (\i,\j){};
					\end{pgfonlayer}
					\begin{pgfonlayer}{E}	
					\clip (-5.5,-2.7) rectangle (4.5,5.3);
					\path[draw] (\i,\j)--(\i+1,\j);
					\path[draw] (\i,\j)--(\i,\j+1);
					\end{pgfonlayer}
				}
				
				\foreach \i in {-3,...,3}
				{
					\begin{pgfonlayer}{drawunder}	
					\clip (-5.5,-2.7) rectangle (4.5,5.3);
					\path[draw,dashed,gray] (-8.5,-8.5+4*\i)--(7.5,7.5+4*\i);
					\path[draw,dashed,gray] (-8.5,7.5+4*\i)--(7.5,-8.5+4*\i);
					\end{pgfonlayer}
				}
                {
					\begin{pgfonlayer}{drawover}	
					\clip (-5.5,-2.7) rectangle (4.5,5.3);
                    
					\node[fill=white,below] at (-2.5,-.5) {$t$};
					\node[fill=white,above] at (-4.5,-2.5) {$t^{\uparrow}$};
                    \foreach \i in {-1,...,1}
					{
					\foreach \j in {-1,...,1}
                    {
                    \node[treevertex] at (4*\i-.5,4*\j-2.5){};
                    \node[treevertex] at (4*\i-2.5,4*\j-.5){};
                    }
                    }
                    \path[draw,very thick, dashed] (-6.5,-4.5)--(-4.5,-2.5);
                    \path[draw,very thick,dashed] (-4.5,-2.5)--(-2.5,-0.5);
                    \path[draw,very thick] (-2.5,-0.5)--(-4.5,1.5);
                    \path[draw,very thick] (-2.5,-0.5)--(-0.5,1.5);
                    \path[draw,very thick] (-0.5,1.5)--(1.5,-0.5);
                    \path[draw,very thick] (1.5,-0.5)--(3.5,1.5);
                    \path[draw,very thick] (-2.5,3.5)--(-0.5,1.5);
                    \path[draw,very thick] (1.5,3.5)--(3.5,1.5);
                    \path[draw,very thick, dashed] (-0.5,-2.5)--(1.5,-4.5);
                    \path[draw,very thick, dashed] (3.5,-2.5)--(5.5,-0.5);
                    
					\end{pgfonlayer}
				}
			}
			\end{tikzpicture}
		\end{minipage}%
		\begin{minipage}{.5\textwidth}
        \centering
			\begin{tikzpicture}[scale=.6, vertex/.style={inner sep=1pt,circle,draw,fill}]
			\pgfdeclarelayer{E}
			\pgfdeclarelayer{V}
			\pgfdeclarelayer{drawunder}
			\pgfdeclarelayer{drawover}
			\pgfsetlayers{drawunder,E,V,drawover}
			\foreach \i in {-7,...,7}
			{
				\foreach \j in {-5,...,5}
				{
					\begin{pgfonlayer}{V}	
					\clip (-5.5,-2.7) rectangle (4.5,5.3);
					\node[vertex] at (\i,\j){};
					\end{pgfonlayer}
					\begin{pgfonlayer}{E}	
					\clip (-5.5,-2.7) rectangle (4.5,5.3);
					\path[draw] (\i,\j)--(\i+1,\j);
					\path[draw] (\i,\j)--(\i,\j+1);
					\end{pgfonlayer}
				}
				
				\foreach \i in {-3,...,3}
				{
					\begin{pgfonlayer}{drawunder}	
					\clip (-5.5,-2.7) rectangle (4.5,5.3);
					\path[draw,dashed,gray] (-8.5,-8.5+4*\i)--(7.5,7.5+4*\i);
					\path[draw,dashed,gray] (-8.5,7.5+4*\i)--(7.5,-8.5+4*\i);
					\end{pgfonlayer}
				}
                {
					\begin{pgfonlayer}{drawover}	
					\clip (-5.5,-2.7) rectangle (4.5,5.3);
                    \path[draw,very thick] (-5,-3)--(-5,-1)--(-4,-1);
                    \path[draw,very thick] (-4,-1)--(-3,-1)--(-3,0)--(-5,0)--(-5,3)--(-4,3)--(-4,1)--(-1,1)--(-1,2)--(-3,2)--(-3,5)--(-2,5)--(-2,3)--(0,3)--(0,1)--(3,1)--(3,2)--(1,2)--(1,5)--(2,5)--(2,3)--(4,3)--(4,0)--(2,0)--(2,-2)--(1,-2)--(1,0)--(-2,0)--(-2,-2)--(-3,-2);
                    \path[draw,very thick] (-3,-2)--(-4,-2)--(-4,-4);

                    \path[draw,very thick] (-5,6)--(-5,4)--(-4,4)--(-4,6);
                    \path[draw,very thick] (-1,6)--(-1,4)--(0,4)--(0,6);
                    \path[draw,very thick] (3,6)--(3,4)--(4,4)--(4,6);
                    
                    \path[draw,very thick] (-1,-3)--(-1,-1)--(0,-1)--(0,-3);
                    \path[draw,very thick] (3,-3)--(3,-1)--(5,-1);
                    \path[draw,very thick] (4,-3)--(4,-2)--(5,-2);
					\end{pgfonlayer}
				}
			}
			\end{tikzpicture}
		\end{minipage}
\caption{The tree $T_t^{\downarrow}$ and one of the corresponding paths in $\Z^2$.}
\label{fig:treepath}
\end{figure}
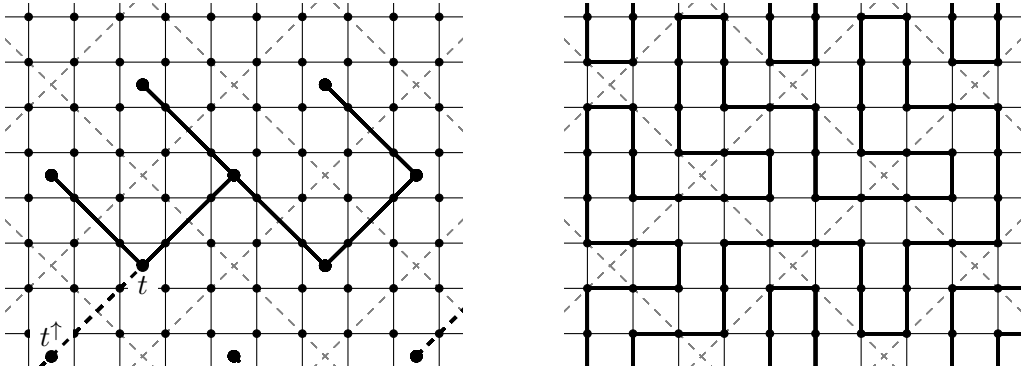
	
	\begin{proof}
		Clearly, swapping colours within any family of attachment squares does not change the degree of any vertex in either colour class. In particular, the subgraphs induced by the resulting colour classes will be $2$-regular, and it only remains to show that they are connected.
		
		For a tile $t$, let $t^{\uparrow}$ be the neighbour of $t$ that lies on the unique ray in $T$ starting at $t$ and let $T^{\downarrow}_t$ be the component of $T - t^{\uparrow}$ which contains $t$. 
		
		We claim that all internal edges of the squares in $T^{\downarrow}_t$ in each colour class are connected by a path contained in the union of the squares in $T^{\downarrow}_t$, see Figure~\ref{fig:treepath}. This is easily proved by induction. If $t$ is a leaf of $T^{\downarrow}_t$ it is obvious. If we inductively assume that it holds for the neighbours of $t$ in $T^{\downarrow}_t$ (whose maximum distance to a leaf in their respective subtree $T^{\downarrow}_s$ is smaller), then it is straightforward to see that the statement also holds for $t$ using the trivial fact that replacing an edge of a cycle by a path results in a connected graph.
		
		Finally note, that every vertex is adjacent to an internal edge of each colour (although these two internal edges might be internal to different tiles). Since for any two tiles $t$ and $t'$, there is a tile $s$ such that both $t$ and $t'$ are contained in $T^{\downarrow}_s$, this implies that each colour class is connected and thus a spanning double ray. 
	\end{proof}
	
	It is well known that $\Z^2$ admits an invariant measure on one ended spanning trees, see for example \cite{Pem91}. If we choose the tree $T$ in the above construction according to such a measure, then it is easy to see that the colour of an edge of $\mathbb Z^2$ only depends on whether or not a specific edge is contained in $T$. Hence cylinder sets are measurable and we end up with a $\Gamma$-invariant decomposition of $\Z ^2$ into spanning double rays.
	
	To finish the proof of Theorem \ref{Z2}, note that $\Gamma$ has finite index in $\mathbb Z^2$. Thus any $\Gamma$-invariant measure $\mu$ on subgraphs of $\Z^2$ can be made into a $\Z ^2$-invariant measure
	\[
	\nu = \frac{1}{[\Z^2 : \Gamma]} \sum_{i} \mu_i,
	\]
	where $\mu_i$ are translations of $\mu$ with respect to a system of representatives of the cosets of $\Gamma$. This shows that $\Z^2$ has an invariant decomposition into double rays, and in particular, $\Z^2$ contains an invariant spanning copy of $\Z$. 

Next, we show that the Cayley graph of $\Z^d$ with respect to the standard generators admits an invariant spanning double ray. First consider the case $d = 3$.
Let $Z_i$ be a copy of the standard Cayley graph of $\Z$ for every $i$. To sample a random spanning double ray $R$ in $X:= Z_1 \times Z_2 \times Z_3$, we first sample a pair of independent spanning double rays $R_{12}$, $R_3$ in $Z_1 \times Z_2$ and $Z_0 \times Z_3$, respectively, each according to the law of the $\Z^2$-invariant spanning double ray we obtained above.\footnote{Alternatively, we could just use the Peano UST curve from \cite{LyonsBook}.}

Let $f$ be one of the two isomorphisms from $Z_0$ to $R_{12}$ mapping $0\in Z_0$ to $0\in Z_1 \times Z_2$, chosen by a fair coin flip. Extend $f$ canonically to a map $\phi: Z_0 \times Z_3 \to R_{12} \times Z_3$ by `multiplying' with $Z_3$. Then $R:= \phi(R_3)$ is a random spanning double ray of $X$.

Clearly, the law of $R$ is invariant under the canonnical action of $\Z$ on $X$ in the $Z_3$ coordinate. We claim that it is also invariant under the canonnical action of $\Z^2$ in the other two coordinates. 

For this, let $r_{12}$ be the (random) labelling of the pairs of vertices $(x,y)$ of $Z_1 \times Z_2$ where $r_{12}(x,y)$ is the distance between $x$ and $y$ along $R_{12}$. Let $B = I_1 \times I_2 \times I_3$, where $I_i$ is a finite subpath of $Z_i$, be a `box' in $X$, to be thought of as a cylinder set of our $\sigma$-algebra. Notice that conditioning on the values of $r_{12}$ inside $I_1 \times I_2$ uniquely determines the distribution of $R \cap B$ by the construction of $R$. As  $R_{12}$ is invariant under the canonnical action of $\Z^2$ on 
$Z_1 \times Z_2$, so is $r_{12}$. Thus the distribution of $R \cap B$ coincides with that of $R \cap gB$ for every $g\in \Z^2$, which completes the proof that $R$ is invariant under the canonical action of $\Z^3$ on $X$.

\medskip
This construction generalises to $\Z^d, d>3$ by induction as follows. Assume we have proved that  the standard Cayley graph of $\Z^{d-1}$ admits both an invariant spanning double ray as well as an invariant spanning copy $H$ of the standard Cayley graph of $\Z^{d-2}$. We repeat the above construction with $X:= Z_1 \times Z_2 \cdots \times Z_d$, with $R_{12}$ as before, but with $R_3$ being a sample of $H$ in $Z_0 \times Z_3 \times Z_4 \times \cdots Z_d$. Repeating the above ideas yields now an invariant spanning double ray in $X$. We now also need an invariant spanning copy of $\Z^{d-1}$ in $X$, to prove our inductive hypothesis, but this is easier: we can just take $R_{12} \times H$.

\medskip
Finally, let $\Gamma$ be an arbitrary finitely generated Abelian group, let $S$ be a finite generating set, and let $G$ be the Cayley graph of $\Gamma$ with generators $S$. Let $S' \subseteq S$ be a maximal set of linearly independent generators with infinite order. Then the subgroup $\Gamma'$ of $\Gamma$ generated by $S'$ has finite index in $\Gamma$, and the Cayley graph $G'$ of $\Gamma'$ with respect to $S'$ is the standard Cayley graph of $\Z^d$ for some $d$.

From now on, view $G'$ as a subgraph of $G$ with the obvious embedding. We will use a \sdr\ in $G'$ to  define a \sdr\ in $G$ as follows. Let $M_1 \cup M_2$ be the decomposition of the edge set of the \sdr\ into two perfect matchings. In order to make the choice canonical, decide by a fair coin flip which of the two sets is $M_1$. It is well known that every Cayley graph of a finite Abelian group admits a Hamilton path, see \cite{MarHamCayley}. From such a path in the Cayley graph of $\Gamma / \Gamma'$ (with generators corresponding to $S$), it is easy to construct a path $P$ in $G$ which starts at $0$ and visits every coset of $\Gamma'$ precisely once. Pick such a path $P$ arbitrarily, let $p$ be the other endpoint of $P$ and let $f_p$ be the automorphism of $G$ mapping $x$ to $x+p$. 

We claim that the subgraph of $G$ consisting of $M_1$, $f_p(M_2)$, and  all translates of $P$ by elements of $\Gamma'$ is the desired \sdr. Since $P$ visits every coset exactly once, every vertex of $G$ lies on a unique translate of $P$. Vertices in $G'$ and $f_p(G')$ are incident to an additional matching edge in $M_1$ and $f_p(M_2)$ respectively, whence $R$ is $2$-regular. To see that it is connected, observe that contracting all edges in translates of $P$ gives the \sdr\ of $G'$ that we started with.

To finish the proof, note that if the \sdr\ in $G'$ was taken from an invariant distribution on \sdr s, then the result of the above construction is invariant under $\Gamma'$. Since $[\Gamma : \Gamma']$ is finite, we can then use the same argument as in the case of $\Z^2$ to turn this into a $\Gamma$-invariant \sdr.

	\section{Proof of Theorem \ref{thm:invariantspanningdoubleray}}
	Let $(\mathcal T_G,\mu)$ be the probability space arising from the implication \ref{itam} $\to$ \ref{itst} of Corollary~\ref{corol} (the condition of unimodularity is not needed for that implication, see the remark after the statement), so that the elements of $\mathcal T_G$ are the spanning trees of $G$, and consider the product probability space 
	\[
	\mathcal T := (\Omega, \Sigma, \mathbb P) = \mathcal T_G \times \prod_{v \in V} \mathcal U_{<_{v}},
	\]
	where  $\mathcal U_{<_{v}}$ denotes the uniform distribution on total orders on the neighbours of $v$ in $G$.
	
	We will construct a map $\phi \colon \Omega \to 2^{E(G^3)}$, mapping each $\omega = (T, (<_{v})_{v \in V}) \in \Omega$ to a spanning double ray of $G^3$. The main idea behind the definition of this map $\phi$ stems from \cite{fleisch}, and is illustrated in Figure~\ref{gcube}. To explain it, let us first pretend that $T$ is a spanning tree of a  finite graph $G$, rooted at a vertex $v$, and we want to use it to obtain a Hamilton cycle of $G^3$. We do this by induction on the size of $T$ as follows. Suppose that for each subtree $T_i$ of $T$ obtained by removing $v$, we have a path $P_i$ in $G^3$ from the root $v_i$ of $T_i$ to one of its neighbours $v_i^\dagger$ in $T_i$ visiting each vertex of $T_i$ exactly once (unless $v_i$ is the unique vertex of $T_i$, in which case $P_i$ is the trivial path). Then we can join the $P_i$ using an edge from $G^3$ as shown in Figure~\ref{gcube}, to obtain a path from $v$ to a neighbour $v^\dagger$, proving the inductive hypothesis. We can the produce a Hamilton cycle of $G^3$ by adding the $v$--$v^\dagger$ edge to that path. 
	
	It is possible to produce such a Hamilton cycle by making local decisions at each vertex of $T$ simultaneously, rather than working recursively: given a fixed ordering of the neighbours of $v$ in Figure~\ref{gcube}, we can decide which additional edges to join the $P_i$ with without seeing the $P_i$. This is the key observation that will allow us to use the same idea in our situation, where $T$ is infinite, with one or two ends. These `local' decisions are encoded in the aforementioned map $\phi$. This is where the orders  $(<_{v})_{v \in V}$ will be used.
	
	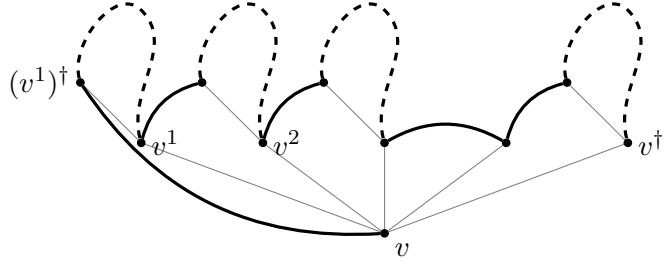
\begin{figure}
		\centering
		\begin{tikzpicture}[scale=.8,vertex/.style={inner sep=1pt,circle,draw,fill}]
		\pgfdeclarelayer{E}
		\pgfdeclarelayer{V}
		\pgfdeclarelayer{labels}
		\pgfsetlayers{E,V,labels}
		\begin{pgfonlayer}{V}	
		\node[vertex] (v) at (0,0){};
		\foreach \i in {-2,...,2}
		\node[vertex] (v\i) at (2*\i,1.5){};
		\foreach \i in {-2,...,0,2}
		\node[vertex] (vd\i) at (2*\i-1,2.5){};
		\end{pgfonlayer}
		\begin{pgfonlayer}{E}	
		\foreach \i in {-2,...,2}
		\path[draw,gray] (v)--(v\i);
		\foreach \i in {-2,...,0,2}
		{
			\path[draw,gray] (v\i)--(vd\i);
			\path[draw,dashed,very thick] (vd\i) to[out=105,in=135] (2*\i,3.7) to[out=-45,in=105] (v\i);
		}
		\path[draw,very thick] (v) to[bend left] (vd-2);
		\path[draw,very thick] (v-2) to[bend left] (vd-1);
		\path[draw,very thick] (v-1) to[bend left] (vd0);
		\path[draw,very thick] (v0) to[bend left] (v1);
		\path[draw,very thick] (v1) to[bend left] (vd2);
		\end{pgfonlayer}
		\begin{pgfonlayer}{labels}
		\node[below right] at (v) {$v$};
		\node[right] at (v-2) {$v^1$};
		\node[right] at (v-1) {$v^2$};
		\node[right] at (v2) {$v^\dagger$};
		\node[left] at (vd-2) {$(v^1)^\dagger$};
		\end{pgfonlayer}
		\end{tikzpicture}
		
		\caption{Edges of types \ref{itm:e-up} and \ref{itm:e-across} together with $v^i$--$(v^i)^\dagger$-paths form a $v$--$v^\dagger$-path.}
		\label{gcube}
	\end{figure}
	
	\medskip
	We define $\phi$ as follows. Let $T$ be a 1- or 2-ended spanning tree of $G$. For a vertex $v$, denote by $N^{\infty}_v$ the (one or two) neighbours of $v$ in $T$ that lie in an infinite component of $T-v$, and let $N^{\not \infty}_v$ be the remaining neighbours of $v$ in $T$. Clearly $<_{v}$ gives rise to a total order on $N^{\not \infty}_v$. Denote by $v^{i}$ the $i$-th largest element in this order. If $N^{\not \infty}_v \neq \emptyset$, then denote by $v^\dagger$ the maximal element, otherwise let $v^\dagger = v$. 
	
	Now let $\phi(\omega) \subseteq E(G^3)$ consist of the following edges:
	\begin{enumerate}[label = (\roman*)]
		\item \label{itm:e-up} 
		for $v \in V(G)$ with $N^{\not \infty}_v \neq \emptyset$ the edge connecting $v$ and $(v^1)^\dagger$ 
		\item \label{itm:e-across}
		for $v \in V(G)$ with $N^{\not \infty}_v \neq \emptyset$ and $v^i <_v v^\dagger$ the edge connecting $v^i$ to $(v^{i+1})^\dagger$, and
		\item \label{itm:e-connect}
		for $uv \in E(G)$ where $u \in N^{\infty}_v$ and $v \in N^{\infty}_u$\\
		(This case only materialises when $T$ contains a double ray, and that double ray contains $uv$.)
		\begin{itemize}[label=-]
			\item if $u = \max N^{\infty}_v$ and $v = \max N^{\infty}_u$ the edge $uv$,
			\item if $u = \max N^{\infty}_v$ and $v \neq \max N^{\infty}_u$ the edge $u^\dagger v$,
			\item if $u \neq \max N^{\infty}_v$ and $v \neq \max N^{\infty}_u$ the edge $u^\dagger v^\dagger$,
		\end{itemize}
	\end{enumerate}
	
	As we will prove below (Lemma~\ref{lemphi}), items \ref{itm:e-up} and \ref{itm:e-across} implement the intuition of Figure~\ref{gcube}, while item  \ref{itm:e-connect} ensures that if $R$ is a double ray in $T^3$ (which is unique if it exists), then for each edge $uv$ in $R$, we choose exactly one edge joining the subtree of $T - uv$ rooted at $u$ to the subtree rooted at $v$, and we choose it in such a way to ensure all vertices have degree 2. If $G$ is 1-ended, then we may assume $T$ to be 1-ended too by a result of Timar \cite{TimOne}, in which case item  \ref{itm:e-connect} can be dropped and the proof of Lemma~\ref{lemphi} accordingly simplified.
	
	We would like to interpret the pushforward of $\mathcal T$ via $\phi$ 
	as a measure on subgraphs of $G^3$. In order to do so, it is desirable that cylinder sets (i.e.\ events telling us what happens on finite subgraphs of $G^3$) are measurable. The following proposition says that this is the case.
	
	\begin{lem}
		For every cylinder set $C$ of $2^{E(G^3)}$, the set $\phi^{-1}(C)$ is measurable in $\mathcal T$. 
	\end{lem}
	
	\begin{proof}
		It is sufficient to show that for every edge $e \in E(G^3)$ the event $[e \in \phi(\omega)]$ is measurable. 
		
		First note that the event 
		\[
		[u \in N^{\not \infty}_v] = [uv \in E(T) \text{ and } \exists V_u  \subseteq V(G)\colon u\in V_u,  |V_u| < \infty, \text{no edge leaving $V_u$ is in $T-uv$}]
		\]
		is $\mathcal T$-measurable for any pair $u,v$. Since there are only countably many finite sets containing $u$, this is a countable union of cylinder sets, and thus $\mathcal T$-measurable. Consequently the events $[u \in N^{\infty}_v]$, $[u=v^i]$, $[u=v^\dagger]$, and $[u=\max N^{\infty}_v]$ are measurable because they can be obtained by finite intersections and unions of events of the form $[xy \in E(T)]$,$[x \in N^{\not \infty}_y]$, $[x \notin N^{\not \infty}_y]$ and $[x<_z y]$. Finally, the event $[e \in \phi(\omega)]$ can be obtained by finite intersections and unions of the above events, hence it is measurable in $\mathcal T$.
	\end{proof}

	Theorem \ref{thm:invariantspanningdoubleray} is now an easy consequence of the following lemma.
	
	\begin{lem} \label{lemphi}
		Let $H$ be the random subgraph of $G^3$ defined by the pushforward of $(\Omega, \Sigma, \mathbb P)$ via $\phi$. Then $H$ almost surely is a spanning double ray.
	\end{lem}
	
	\begin{proof}
		We prove the statement of the proposition for a fixed tree $T$. This is sufficient since $\mathbb P (A) = \int \mathbb P(A \mid T) \, d\mu(T)$ where $\mu$ is the measure on the space $\mathcal T_G$.
		
		For a vertex $v$ denote by $T^{\not\infty}_v$ the subtree of $T$ induced by $v$ and all finite components of $T-v$. Let $h(v)$ the maximum distance from $v$ to a vertex in $T^{\not\infty}_v$. We prove by induction on $h$ that the subgraph of $H$ induced by the vertices of $T^{\not\infty}_v$ is a spanning path from $v$ to $v^\dagger$. If $h(v) = 0$, this is trivially true. For the induction step, note that there can't be an edge of type \ref{itm:e-connect} inside $T^{\not\infty}_v$. Indeed, if $xy$ is such an edge, then there is an edge $e$ separating $x$ and $y$ in $T$ such that neither component of $G - e$ is finite which implies that $x \notin T^{\not\infty}_y$. Now the induction hypothesis together with the definition of edges of type \ref{itm:e-up} and \ref{itm:e-across} prove the claim.
		
		In particular, this claim shows that all vertices of $T^{\not\infty}_v$ lie in the same component of $H$ and that apart from $v$ and $v^\dagger$ all of them have degree $2$. 
		
		If $T$ has one end, then for any vertex $x$ there is a vertex $v$ such that $x \in T^{\not\infty}_v \setminus\{v,v^\dagger\}$, showing that $x$ has degree $2$. Furthermore, for any two vertices $x,y$ there is a vertex $v$ such that $T^{\not\infty}_v$ contains both $x$ and $y$, showing that $H$ is connected.
		
		If $T$ has two ends, then the edges for which both components of $T - e$ are infinite form a double ray. Note that these are precisely the edges $uv$ for which $u \in N^{\infty}_v$ and $v \in N^{\infty}_u$. Let $x$ be a vertex on this double ray, and let $y$ and $z$ be its neighbours. Then the edges of type \ref{itm:e-connect} corresponding to $xy$ and $xz$ connect $T^{\not\infty}_x$ to $T^{\not\infty}_y$ and $T^{\not\infty}_z$ respectively. In particular, this shows that $H$ is connected. Finally note that one of the edges of type \ref{itm:e-connect} attaches to $x$ and the other one attaches to $x^\dagger$ (depending on whether $y <_x z$ or $z <_x y$), thus making sure that both $x$ and $x^\dagger$ have degree $2$ in $H$.
	\end{proof}
	
\section*{Acknowledgements}
We thank Gabor Pete for a discussion that triggered this paper.
	
	\bibliographystyle{plain}
	\bibliography{collective}
\end{document}